\documentclass[a4paper,11pt]{amsart}

\usepackage{geometry}
\usepackage{soul}

\usepackage{xpatch}
\xpatchbibdriver{article}
  {\newunit\newblock\usebibmacro{addendum+pubstate}}
  {\setunit{\addcomma\space}\usebibmacro{addendum+pubstate}}
  {}{}

\usepackage{amsfonts, amssymb, amscd}
\numberwithin{equation}{section}

\usepackage{csquotes}
\usepackage[style=alphabetic,max
names=7,giveninits=true,url=false]{biblatex}

\AtBeginBibliography{\small}

\DeclareFieldFormat[article,inproceedings,incollection]{title}{#1}
\DeclareFieldFormat[article,inproceedings,incollection]{title}{\mkbibemph{#1}}
\DeclareFieldFormat[article]{journaltitle}{#1} 
\renewbibmacro{in:}{} 

\DeclareFieldFormat[article]{volume}{\textbf{#1}}
\DeclareFieldFormat[book]{volume}{\textbf{#1}}

\DeclareFieldFormat{eprint:arxiv}{%
  \ifhyperref
    {\href{https://arxiv.org/abs/#1}{arXiv\addcolon\space\nolinkurl{#1}}}
    {arXiv\addcolon\space\nolinkurl{#1}}%
}

\defbibenvironment{bibliography}
  {\list
     {}
     {%
       \setlength{\leftmargin}{\bibhang}      
       \setlength{\itemindent}{-\leftmargin}  
       \setlength{\itemsep}{\bibitemsep}
       \setlength{\parsep}{\bibparsep}
     }}
  {\endlist}
  {\item
   \printtext[labelalphawidth]{%
     \printfield{labelprefix}%
     \printfield{labelalpha}%
     \printfield{extraalpha}}%
   \space}

\AtEveryBibitem{%
  \ifentrytype{online}
    {} %如果是 online 类型，跳过处理
    {%
      \iffieldundef{eprint}
        {\clearfield{url}\clearfield{urldate}\clearfield{doi}}
        {\clearfield{eprintclass}} 
    }%
}

\renewbibmacro*{addendum+pubstate}{%
  \setunit{\addcomma\space}% <--- 关键修改：强制设置为逗号
  \printfield{addendum}%
  \newunit\newblock
  \printfield{pubstate}}

\usepackage{bm}
\usepackage{verbatim}
\usepackage{mathrsfs}
\usepackage{graphicx}
\usepackage{tikz-cd}
\usepackage{subcaption}
\usepackage{listings}
\usepackage{subfiles}
\usepackage[toc,page]{appendix}
\usepackage{mathtools}
\usepackage{comment}
\usepackage{enumerate}
\usepackage{enumitem}
\usepackage[all]{xy}

\usepackage{graphicx}
\graphicspath{{images/}}
\usepackage{appendix}
\usepackage{hyperref}
\hypersetup{
    colorlinks=true,
    citecolor=red,
    linkcolor=blue,
    filecolor=magenta,      
    urlcolor=red,
}
\title{Volumes of foliations birationally bounded by algebraically integrable families}
\author{Zhixiu Fan}

\address{School of Mathematical Sciences, Fudan University, Shanghai 200433, China}
\email{23110180008@m.fudan.edu.cn}

\date{}

\DeclareMathOperator{\vol}{vol}
\DeclareMathOperator{\Exc}{Exc}
\DeclareMathOperator{\Supp}{Supp}

\newtheorem{thm}{Theorem}[section]

\newtheorem{lem}[thm]{Lemma}
\newtheorem{prop}[thm]{Proposition}

\newtheorem{claim}[thm]{Claim}

\theoremstyle{definition}
\newtheorem{defn}[thm]{Definition}
\newtheorem{ques}[thm]{Question}
\theoremstyle{definition}
\newtheorem{rem}[thm]{Remark}

\newtheorem{nota}[thm]{Notation}

\newtheorem*{introthm}{Theorem}

\theoremstyle{definition}

\newcommand\CS{{\mathcal{S}}}
\newcommand\CO{{\mathcal{O}}}

\newcommand\CF{{\mathcal{F}}}
\newcommand\CG{{\mathcal{G}}}
\newcommand{\bR}{{\mathbb R}}
\newcommand{\bC}{{\mathbb C}}
\newcommand{\bQ}{{\mathbb Q}}

\newcommand{\bir}{\dashrightarrow}
\newcommand{\rk}{\operatorname{rank}}

\newcommand{\mc}{%
  M\raisebox{\dimexpr\fontcharht\font`M - \height}{%
    \scriptsize c%
  }%
}

\begin{document}

\begin{abstract}
    We prove that for log canonical foliations which are birationally bounded by algebraically integrable families, the set of their volumes satisfies the DCC. This answers a special case of a question posed by Cascini, Hacon, and Langer. As a key ingredient, we establish the deformation invariance of relative log canonical volumes for a family of weak semistable morphisms, which can be viewed as a relative version of the classical result proved by Hacon, \mc Kernan, and Xu.
\end{abstract}

\maketitle

\tableofcontents

\section{Introduction}
We work over the field of complex numbers $\bC$. DCC stands for the \emph{descending chain condition}.

 The volume of a projective variety is one of the most intensively studied invariants in birational geometry. For varieties of general type, a famous result of Hacon, \mc Kernan, and Xu establishes that their volumes satisfy the DCC, and in particular, have a positive lower bound. They also prove the effective birationality for the log canonical systems for such varieties. More precisely:
\begin{introthm}[{\cite[Theorem 1.3]{HMXacc}}]\label{HMX dcc} Fix a positive integer $d$ and a DCC set $I\subset{[0,1]}$. Let $\mathcal{D}$ be the set of log canonical pairs $(X,\Delta)$ such that $\dim X=d$, the coefficients of $\Delta$ belong to $I$, and $K_X+\Delta$ is big. Then we have
\begin{enumerate}
    \item the set $\{\vol(X,K_X+\Delta)\,|\, (X,\Delta)\in \mathcal{D}\}$ satisfies the DCC;
    \item (effective birationality) there exists a positive integer $m$ depending only on $d$ and $I$, such that the linear system $|m(K_X+\Delta)|$ defines a birational map.
\end{enumerate}
\end{introthm}
In this paper, we are interested in the volume of algebraically integrable foliations.

%%%%%%%%%%%%%%
\medskip
\noindent\textbf{Volume of Algebraically Integrable Foliations.} In recent years, the minimal model program (MMP) for foliations has been extensively developed, starting with \cite{mcqui08}, and then \cite{ACSS21,CS21,CS251,CHLX23,LMX24}, especially for algebraically integrable ones.

 It is interesting to ask whether the above theorem of Hacon--\mc Kernan--Xu holds for foliations. However, the analogue of effective birationality fails in general, even for foliations induced by well-behaved fibrations, as shown in \cite[Theorem 1.3]{lu25}. In \cite[Theorem 2]{HL21}, Hacon and Langer showed that for a canonical model $(X,\CF)$ of a foliated surface of general type, the effective birationality holds if we fix the Hilbert polynomial $P(m)=\chi\left(X,\CO_{X}\left(mK_\CF\right)\right)$. See also \cite{Chen21,SS23,PS19,passantino24} for more progress.
 
 Hacon and Langer \cite[Question 4]{HL21} asked if volumes of canonical models of foliated surfaces are well ordered (e.g., satisfy the DCC) and admit a positive minimum. It was generalized to higher dimensions by Cascini \cite[Section 3]{cas21}. 

\begin{ques}[J.~Pereira,\,\, {\cite[Section 3]{cas21}, \cite[Question 4]{HL21}}]\label{Main Question fol}
    Let $d$ be a positive integer. Let $V=\{\vol(X,K_\CF)\}$ where $(X,\CF)$ is a  projective foliated pair with canonical singularities of general type (i.e. $K_\CF$ is big). Does $V$ satisfy the DCC or does it admit a positive minimum?
\end{ques}
Question \ref{Main Question fol} remains open even for surface foliations. See \cite{PS19,Chen21,HL21,LT22} for related works. In practice, it is often natural to consider the ``log version" of Question \ref{Main Question fol}. That is, we work with foliated triples rather than pairs, and replace  ``canonical singularities'' with ``lc singularities'' (see Definition \ref{defn: foliation singularity}).
In the recent work \cite{HJLL24}, it was shown that the volumes of lc algebraically integrable foliations belong to a discrete set depending only on their rank and the volumes of their general leaves. For stable families over a curve with klt general fibres and boundary coefficients in a finite set, the existence of a positive lower bound for the volume was proved in \cite[Corollary 1.8]{CPT25}.
 
In this paper we study Question \ref{Main Question fol} for algebraically integrable foliations via the MMP developed in \cite{ACSS21,CHLX23,LMX24}, etc. We consider a special case of Question \ref{Main Question fol} where foliations are \textbf{birationally bounded by algebraically integrable families} (cf. Definition \ref{bir bdd fol triples}). This condition is stronger than the usual assumption of birationally boundedness. We require that both the foliations and their general leaves are birationally bounded (see Remark \ref{rem: bir bdd foliations}). The main result of this paper is the following, which can be viewed as a foliated version of \cite[Theorem 1.9]{HMX1}.

\begin{thm}\label{main thm last}
Fix a DCC set $I\subset\bR^{\ge 0}$. Let $\CS$ be a set of lc algebraically integrable foliated triples $(Y,\CG,C)$ such that the coefficients of $C$ belong to $I$. Assume that $\CS$ is birationally bounded by algebraically integrable families (cf. Definition \ref{bir bdd fol triples}). 

Then the set $$\{\vol(Y,K_{\CG}+C)\,|\,(Y,\CG,C)\in \CS\}$$satisfies the DCC.
\end{thm}

%%%%%%%%%%%(invariance of plurigenera)
\noindent\textbf{Deformation Invariance of Plurigenera and Volumes.} A deep result of Siu \cite{Siu98,Siu02} states that the plurigenera of a family of smooth projective varieties are deformation invariant. Hacon, \mc Kernan, and Xu generalized this to the case of log pairs \cite[Theorem 1.8]{HMX1}. Specifically, let $(X,B)\to T$ be a projective morphism from a log canonical pair $(X,B)$ to a smooth variety $T$, such that $(X,B)$ is log smooth over $T$ (cf. Definition \ref{def: log smooth over base}). They proved that 
the log plurigenera $h^0(X_t,m(K_{X_t}+B_t))$ are independent of $t\in T$, assuming that $(X,B)$ is klt and $K_X+B$ is big over $T$. As a direct corollary, the log canonical volume $\vol(X_t,K_{X_t}+B_t)$ is also independent of $t\in T$. Notably, this volume invariance holds even without the klt and bigness assumptions. For more general results, see \cite[Theorem 4.2]{HMX1} and \cite[Corollary 1.3 and Lemma 6.3]{HMX2}.

 As a key ingredient in the proof of Theorem \ref{main thm last}, we establish the following result, which can be viewed as a relative version of the volume invariance discussed above.

\begin{thm}[see Theorem \ref{inv vol} for a more general form]\label{inv vol semistable} Let $(X,B)$ be a log pair. Suppose that $f:X\to Z/T$ is a contraction over a smooth variety $T$, such that
 \begin{enumerate}
     \item the morphism $f:(X,\Sigma_X)\to (Z,\Sigma_Z)$ is semistable with respect to the divisors $\Sigma_X$ and $\Sigma_Z$,
     \item $Z\to T$ is a contraction and $(Z,\Sigma_Z)$ is log smooth over $T$,
     \item $0\le B\le \Sigma^h_X$, where $\Sigma^h_X$ denotes the horizontal$/Z$ part of $\Sigma_X$.
 \end{enumerate}
Then the relative log canonical volume $\vol(X_t,K_{X_t/Z_t}+B_t)$ is independent of $t\in T$.
\end{thm}
For the definition of a semistable morphism, we mainly follow \cite{AK00} with minor modification. See Definition \ref{def: wss} and Remark \ref{rem of def of wss}. In Theorem \ref{inv vol semistable}, if we set $Z=T$, the condition that $(Z,\Sigma_Z)$ is log smooth over $T$ implies $\Sigma_Z=0$, and thus $(X,\Sigma_X)$ is log smooth over $T$. Consequently, this reduces to the volume invariance result of Hacon--\mc Kernan--Xu described above.

We further explain the conditions in Theorem \ref{inv vol semistable}. Consider the foliation $\CF$ induced by $f$. The assumption $0\le B\le\Sigma^h_X$ ensures that the foliated triple $(X,\CF,B)$ has log canonical singularities (see Definition \ref{defn: foliation singularity} and Remark \ref{fol log sm is lc}). Therefore, this theorem can also be viewed as a version of invariance of log canonical volume for foliations. In this direction, see \cite{CF18} for some results in the surface case. Here we mention that, although the volume is invariant under the assumptions of Theorem \ref{inv vol semistable}, we do not actually show the invariance of plurigenera during the proof. We may therefore propose the following:
\begin{ques}\label{question:inv of plurigenera}
    Notation as in Theorem \ref{inv vol semistable}. Further assume that $K_{X/Z}+B$ is big over $T$. Is $h^0(X_t,m(K_{X_t/Z_t}+B_t))$ independent of $t\in T$ for any sufficiently divisible positive integer $m$? 
\end{ques}
\begin{rem}
    Note that in Question \ref{question:inv of plurigenera}, the assumption that $m$ is sufficiently divisible is necessary, as the invariance of plurigenera fails for small values of $m$, even if $f:X\to Z$ is a smooth morphism (see \cite[Example 3.10]{CF18}).
\end{rem}

%%%%%%%%%(Sketch)
\noindent\textbf{Structure of the paper.} The paper is organized as follows. In Section \ref{preliminaries}, we gather some basic notation and tools used in the paper. In Section \ref{section: dcc for fixed model}, we prove a special case of Theorem \ref{main thm last}, where the foliated triples are birational to a fixed foliated log smooth model (cf. Proposition \ref{dcc of fixed model}). In Section \ref{section4}, we prove the invariance of relative log canonical volumes for a family of weak semistable morphisms (cf. Theorem \ref{inv vol}), which is a more general form of Theorem \ref{inv vol semistable}. Finally in Section \ref{section: main thm}, we prove the main result Theorem \ref{main thm last}, by combining the technique of weak semistable reduction \cite{AK00} (cf. Theorem \ref{weak semistable reduction}) with Proposition \ref{dcc of fixed model} and Theorem \ref{inv vol}.

%%%%%%%%%%
\section{Preliminaries}\label{preliminaries}
Throughout this paper, we mainly work with normal quasi-projective varieties to ensure consistency with the references. We adopt the standard notation and definitions in \cite{KM98, BCHM10}. For foliations, we generally follow the notation and definitions in \cite{CS21, ACSS21, CHLX23, CS251}, but there may be minor differences.
\subsection{Special notation}

\begin{nota}
    In this paper, a \emph{contraction} is a projective morphism between quasi-projective normal varieties $f:X\to Y$ such that $f_*\CO_X=\CO_Y$. In particular, $f$ has connected fibres. Let $f:X\to Z$ be a contraction, and $D$ an $\bR$-divisor on $X$. We define the \emph{horizontal$/Z$ part} $D^h$ of $D$ as the part of $D$ supported on prime divisors dominating $Z$; the \emph{vertical$/Z$ part} of $D$ is $D^v:=D-D^h$. A \emph{finite cover} is a finite surjective morphism. 
    
    Let $f:X\dasharrow X'$ be a birational map between normal varieties. We denote by $\Exc(f)$ the reduced divisor supported on the codimension one part of the exceptional locus of $f$. 
    
    Let $X$ be a normal variety, and let $M$ be an $\bR$-divisor on $X$. We denote the coefficient of a prime divisor $D$ in $M$ by $\mu_D M$.

The notation  ``/'' is used as shorthand for ``over''. For example, ``$/Z$'' means ``over $Z$''.

\end{nota}
    
\begin{defn}[Relative log smoothness]\label{def: log smooth over base}
	Given a log pair $(X,B)$ and a morphism $X\to T$, we say that $(X,B)$ is \emph{log smooth over $T$} if $(X,B)$ is log smooth and both $X$ and every stratum of $(X,B)$ are smooth over $T$. Here, a \emph{stratum} of $(X,B=\sum_{i=1}^{n}b_i B_i)$ refers to an irreducible component of the intersection $\bigcap_{i\in I}B_i$, where $I\subset \{1,...,n\}$ is a nonempty subset.
\end{defn}
\begin{defn}
    Let $X$ be a normal projective variety of dimension $d$, and $D$ an $\bR$-divisor on $X$. The \emph{volume} of $D$ is defined as: $$\vol(X,D):=\limsup_{m\to \infty}{\frac{h^0(X,\lfloor mD\rfloor)}{m^d/d!}}.$$ 
\end{defn}
For more background, see \cite{Laz04}.

%%%%%%%%%%%%
\subsection{Toroidal geometry and weak semistable reduction}
We recall some basic facts of toroidal embedding, following the standard notation and definitions in \cite{KKMS73} and \cite{AK00}. We also adopt the notion of \emph{toroidal couples} (cf. \cite[Section 3.7]{Bir23}).
 See also \cite{Qu25} for a more detailed treatment. 
 We refer to \cite{Cox11} for the general theory of toric varieties. 
 
 Recall that a \emph{couple} $(X,\Sigma_X)$ consists of a normal variety $X$ and a reduced Weil divisor $\Sigma_X$ on $X$. A morphism $(X,\Sigma_X)\to (Y,\Sigma_Y)$ of couples is a morphism $f:X\to Y$ between normal varieties such that $f^{-1}(\Sigma_Y)\subset{\Sigma_X}$ (cf. \cite[Section 3.6]{Bir23}).

\begin{defn}[Toroidal couple and toroidal embedding, {cf. \cite[II \S 1, Definition 1]{KKMS73} and \cite[Definition 1.2]{AK00}}]
 Let $(X,\Sigma_X)$ be a couple. We say that the couple is \emph{toroidal} if for any closed point $x\in X$, there exists a normal affine toric variety $W$ together with a closed point $w\in W$, such that there is an isomorphism of local complete $k$-algebras
$$\phi:\widehat{\mathcal{O}}_{X,x}\to\widehat{\mathcal{O}}_{W,w}$$ mapping the ideal of $\Sigma_X$ to the ideal of the toric boundary $\Sigma_W\subset W$ (that is, the reduction of the complement of the torus $\mathbb{T}_W$ in $W$). We call the data $\{(W,\Sigma_W),w\}$ a \emph{local toric model} of $\{(X,\Sigma_X),x\}$.

Let $(X,\Sigma_X)$ be a toroidal couple, and let $U_X:=X\backslash \Supp \Sigma_X$, which is an open subset of $X$. We call $(U_X\subset X)$ a \emph{toroidal embedding}, see \cite[page 54]{KKMS73}. 
\end{defn}

\begin{rem}\label{toroidal lc}
     If $(X,\Sigma_X)$ is a toroidal couple, then $X$ is normal and Cohen--Macaulay, $K_X+\Sigma_X$ is Cartier, and $(X,\Sigma_X)$ is an lc pair (see \cite[Lemma 3.8]{Bir23}). Moreover, if $X$ is smooth, then $(X,\Sigma_X)$ is log smooth. 
\end{rem}

\begin{nota}[Strict toroidal, strata]\label{strict toroidal}
    Let $(X,\Sigma_X)$ be a toroidal couple. We say that $(X,\Sigma_X=\sum_{i=1}^{n}D_i)$ is \emph{strict toroidal} if every irreducible component of $\Sigma_X$ is normal (cf. \cite[page 57]{KKMS73}). In this paper, all the toroidal couples $(X,\Sigma_X)$ are assumed to be \textbf{strict}.
    
    In this case, $\bigcap_{i\in I}D_i$ is normal and $\bigcap_{i\in I}D_i\backslash\bigcup_{i\notin I}D_i$ is smooth for each subset $I\subset\{1,...,n\}$ (cf. \cite[page 57]{KKMS73}). The irreducible components of $\bigcap_{i\in I}D_i$ are called the \emph{strata} of $(X,\Sigma_X)$, and are therefore normal. If $X$ is clear from the context, we will call them the strata of $\Sigma_X$. Moreover, if a stratum $S$ of $\Sigma_X$ is an irreducible component of the intersection of components of a reduced divisor $D$ (where $D\le\Sigma_X$), we also call $S$ a stratum of $D$.

    Let $(X,\Sigma_X)$ be a (strict) toroidal couple. The non-klt centres of this pair are exactly its strata. Assume that $0\le B\le \Sigma_X$, $K_X+B$ is $\bQ$-Cartier, and let $V$ be a non-klt centre of $(X,B)$. Then the adjunction formula $K_V+B_V=(K_X+B)|_V$ holds. Moreover, if we write $K_V+\Sigma_V=(K_X+\Sigma_X)|_V$, then $(V,\Sigma_V)$ is also toroidal (cf. \cite[Section 2.12]{Bir21}).

\end{nota}

\begin{defn}[Quasi-smooth toroidal couple, {cf. \cite{AK00}}]\label{quasi-smooth}A toroidal couple $(X,\Sigma_X)$ is \emph{quasi-smooth} if it has quotient toric singularities. That is, locally in the analytic topology, it is the quotient of a smooth variety by a finite Abelian group.
\end{defn}

\begin{rem}\label{quasi-sm iff q-factorial}
    Each toroidal couple $(X,\Sigma_X)$ can be associated with a rational conical polyhedral complex $\Delta_X$ (see \cite[Section 1.3]{AK00}). $(X,\Sigma_X)$ is quasi-smooth if and only if $\Delta_X$ is simplicial, which is equivalent to $X$ being $\bQ$-factorial. To prove the latter, one may use the notion of \emph{\'etale-toric charts} (see \cite[page 195]{KKMS73} and \cite[Theorem 2.13]{Qu25})
    to reduce to the toric case.
\end{rem}

\begin{defn}[Toroidal morphism, {cf. \cite[Definition 1.3]{AK00}}] Let $f:(X,\Sigma_X)\to (Y,\Sigma_Y)$ be a dominant morphism of couples.  $f$ is called a \emph{toroidal morphism} if for every closed point $x$, there exist local toric models $\{(W,\Sigma_W),w\}$ and $\{(V,\Sigma_V),v\}$ of $\{(X,\Sigma_X),x\}$ and $\{(Y,\Sigma_Y),y=f(x)\}$, respectively, together with a toric morphism $W\to V$ of toric varieties inducing a commutative diagram
\begin{equation*}
    \begin{aligned}\label{formal-defn-diagram}
    \xymatrix{
    \widehat{\mathcal{O}}_{{X},{x}}\ar[r] & \widehat{\mathcal{O}}_{{W},w}\\
    \widehat{\mathcal{O}}_{{Y},{y}} \ar[u] \ar[r] & \widehat{\mathcal{O}}_{{V},v} \ar[u]
    }\end{aligned}
\end{equation*}
 Here, the vertical maps are induced by the given morphisms and the horizontal maps are isomorphisms induced by the local toric models. 
 
 Suppose $f:(X,\Sigma_X)\to(Y,\Sigma_Y)$ is a toroidal morphism with $(Y,\Sigma_Y)$ log smooth. Then $f$ is equi-dimensional if and only if $f$ is flat (cf. \cite[Remark 4.6]{AK00}). Moreover, the composition of toroidal morphisms is also toroidal (cf. \cite[Corollary 1.6]{AK00}).
\end{defn}

\begin{nota}[Toroidal divisor and toroidal extraction]\label{toroidal extraction}
    Let $(X,\Sigma_X)$ be a quasi-smooth toroidal couple. Let $E$ be an exceptional prime divisor over $X$ with discrepancy $a(E,X,\Sigma_X)<0$. Then by Remark \ref{toroidal lc}, $a(E,X,\Sigma_X)=-1$, and the centre of $E$ on $X$ is a stratum of $\Sigma_X$. In this case we say that $E$ is \emph{toroidal with respect to $(X,\Sigma_X)$}.
 
    There is an extremal birational contraction $p:X'\to X$ which extracts $E$ but no other divisors. This morphism is a toroidal modification in the sense of \cite[Section 1.4]{AK00}. Let $K_{X'}+\Sigma_{X'}$ be the pullback of $K_X+\Sigma_X$. Then $(X',\Sigma_{X'})$ is also quasi-smooth toroidal \cite[page 90, Theorem 6*]{KKMS73}. Moreover, $(X',\Sigma_{X'})$ is unique with these properties. We call this $p:(X',\Sigma_{X'})\to (X,\Sigma_X)$ the \emph{toroidal extraction of $E$}.
\end{nota}    
We will need the following two lemmas in the proof of Theorem \ref{inv vol semistable}, Section \ref{section4}.
\begin{lem}[Restriction of toroidal morphism]\label{lem: restriction of toroidal}
    Suppose $f:(X,\Sigma_X)\to(Y,\Sigma_Y)$ is a flat toroidal morphism such that $(Y,\Sigma_Y)$ is log smooth. Let $D\subset\Sigma_X$ be a horizontal$/Y$ component of $\Sigma_X$. Then the restriction $f_D:(D,\Sigma_D)\to (Y,\Sigma_Y)$ is also a flat toroidal morphism. Moreover, if $f$ has reduced fibres, then so does $f_D$.  
\end{lem}
\begin{proof}
    $(D,\Sigma_D)$ is obtained by adjunction (Notation \ref{strict toroidal}). It suffices to consider the toric morphism of local toric models $f:X_\sigma\to X_\tau$. By assumption $D$ corresponds to a ray of the cone $\Sigma$ of $X_\sigma$ and maps to the zero cone of $X_\tau$ (\cite[Lemma 4.1]{AK00}). By \cite[Lemma 3.3.21]{Cox11}, the restriction of $f$ to $D$ is a well-defined toric morphism. Again by \cite[Lemma 4.1]{AK00}, it is flat, since it maps cones to cones. Moreover, if $f$ has reduced fibres, then we can apply \cite[Lemma 5.2]{AK00} to conclude that $f_D$ also has reduced fibres.
 \end{proof}

\begin{lem}\label{extract hor wss}
    Suppose $f:(X,\Sigma_X)\to(Y,\Sigma_Y)$ is a flat toroidal morphism such that $(Y,\Sigma_Y)$ is log smooth. Let $E$ be a toroidal divisor with respect to $(X,\Sigma_X)$, such that the centre of $E$ on $X$ is horizontal$/Y$. Let $p:(X',\Sigma_{X'})\to (X,\Sigma_X)$ be the toroidal extraction of $E$. Then the induced morphism $f':(X',\Sigma_{X'})\to(Y,\Sigma_Y)$ is also a flat toroidal morphism. Moreover, if $f$ has reduced fibres, then so does $f'$.
\end{lem}
\begin{proof}
    Clearly, $f'$ is a toroidal morphism. Since the centre of $E$ on $X$ is horizontal$/Y$, $f'$ is equi-dimensional (cf. \cite[Remark 4.2]{AK00}) and hence flat. Assume that $f$ has reduced fibres. Then by \cite[Lemma 5.2]{AK00}, $f'$ also has reduced fibres as $E$ is horizontal$/Y$.
\end{proof}

\begin{defn}[Weak semistable, {cf. \cite{AK00}}]\label{def: wss}
   A contraction $f:X\to Z$ between normal varieties is called \emph{weak semistable} if there exist divisors $\Sigma_X$ on $X$ and $\Sigma_Z$ on $Z$ such that 
	\begin{enumerate}
	    \item $f$ is a toroidal morphism between the toroidal couples $(X,\Sigma_X)$ and $(Z,\Sigma_Z)$,
		
		\item $f$ is flat,
		
		\item all fibres of $f$ are reduced, and
		
		\item $(X,\Sigma_X)$ is quasi-smooth and $(Z,\Sigma_Z)$ is log smooth. 
		
	\end{enumerate}
In particular, the vertical$/Z$ part of $\Sigma_X$ is equal to $f^*(\Sigma_Z)$. We also say that $f:(X,\Sigma_X)\to(Z,\Sigma_Z)$ is weak semistable with respect to the divisors $\Sigma_X$ and $\Sigma_Z$. If $X$ is also smooth, we say that the morphism $f:X\to Z$ is \emph{semistable}. 
\end{defn}
\begin{rem}\label{rem of def of wss}
    The above definition is slightly different from \cite[Definition 0.1]{AK00}. First, we require that $(X,\Sigma_X)$ is quasi-smooth. This condition can be achieved by a further small toroidal birational modification (cf. \cite[Remark 4.5]{AK00}). Second, the original definition in \cite{AK00} requires that $f(\Sigma_X)\subset{\Sigma_Z}$ (that is, $\Sigma_X$ has no horizontal$/Z$ components), but we do not impose this condition. This definition is the same as \cite[Definition 3.7]{Jiao25}. See also \cite[Definition 2.19]{EH25} and \cite[Definition 2.2]{ALT18}.
\end{rem}
The following result on weak semistable reduction, established in \cite{AK00}, will be used in the proof of Theorem \ref{main thm last}.

\begin{thm}[Weak semistable reduction, {cf. \cite{AK00}}]\label{weak semistable reduction}
Let $f:X\to Z$ be a contraction between normal varieties, and let $B\subset X$ be a divisor. Then there exists a proper surjective, generically finite morphism $b: Z'\to Z$, a birational contraction $a: X'\to (X\times_Z Z')^m$, and divisors $\Sigma_{X'}\subset{X'}, \Sigma_{Z'}\subset{Z'}$, such that 
\begin{itemize}
    \item $a^{-1}(B\times_Z 
    Z')\subset{\Sigma_{X'}}$, and
    \item the morphism $f':(X',\Sigma_{X'})\to (Z',\Sigma_{Z'})$ is weak semistable.
\end{itemize}
Here, $(X\times_Z Z')^m$ denotes the normalization of the main component of $X\times_Z Z'$. 
\end{thm}
We call $f':(X',\Sigma_{X'})\to (Z',\Sigma_{Z'})$ the \emph{weak semistable model} of $f: (X,B)\to Z$.
%%%%%%%%%%%%%
\subsection{Foliations}

\begin{defn}[Foliations, {cf. \cite{ACSS21,CS21}}]\label{defn: foliation}
Let $X$ be a normal variety. A \emph{foliation} on $X$ is a coherent sheaf $\CF\subset T_X$ such that
\begin{enumerate}
    \item $\CF$ is saturated in $T_X$, i.e. $T_X/\CF$ is torsion free, and
    \item $\CF$ is closed under the Lie bracket.
\end{enumerate}

The \emph{rank} of the foliation $\CF$ is the rank of $\CF$ as a sheaf and is denoted by $\rk\CF$. The \emph{canonical divisor} of $\CF$ is a divisor $K_\CF$ such that $\mathcal{O}_X(-K_{\mathcal{F}})\cong\mathrm{det}(\CF)$. If $\CF=0$, then we say that $\CF$ is a \emph{foliation by points}. Given any dominant map $h: Y\dashrightarrow X$, we denote by $h^{-1}\CF$ the \emph{pullback} of $\CF$ on $Y$ as constructed in \cite[3.2]{Dru21} and say that $h^{-1}\CF$ is \emph{induced by} $\CF$. Given any birational map $g: X\dashrightarrow X'$, we denote by $g_*\CF:=(g^{-1})^{-1}\CF$ the \emph{pushforward} of $\CF$ on $X'$. We say that $\CF$ is an \emph{algebraically integrable foliation} if there exists a dominant map $f: X\dashrightarrow Z$ such that $\CF=f^{-1}\CF_Z$, where $\CF_Z$ is the foliation by points on $Z$, and we say that $\CF$ is \emph{induced by} $f$. 

A subvariety $S\subset X$ is called \emph{$\CF$-invariant} if for any open subset $U\subset X$ and any section $\partial\in H^0(U,\CF)$, we have $\partial(\mathcal{I}_{S\cap U})\subset \mathcal{I}_{S\cap U}$, 
where $\mathcal{I}_{S\cap U}$ is the ideal sheaf of $S\cap U$. For any prime divisor $P$ on $X$, we denote $\epsilon_{\CF}(P):=1$ if $P$ is not $\CF$-invariant and $\epsilon_{\CF}(P):=0$ if $P$ is $\CF$-invariant. For any prime divisor $E$ over $X$, we define $\epsilon_{\CF}(E):=\epsilon_{\CF_Y}(E)$ where $h: Y\dashrightarrow X$ is a birational map such that $E$ is on $Y$ and $\CF_Y:=h^{-1}\CF$. For any $\bR$-divisor $D$ on $X$, we define $$D^\CF:=\sum \epsilon_{\CF}(E)E.$$ where the sum runs over the components of $D$. 
\end{defn}

\begin{rem}
    Let $f:X\to Z$ be an equi-dimensional morphism with reduced fibres, $\CF$ is the foliation induced by $f$, then $K_\CF\sim K_{X/Z}$ by \cite[Notation 2.7 and \S2.9]{Dru17}. In the rest of this paper, we will use this fact several times.
\end{rem}

\begin{defn}[Foliated triples]\label{defn: foliated triple}
 A \emph{foliated triple} $(X,\CF,B)$ consists of a normal quasi-projective variety $X$, a foliation $\CF$ on $X$, and an $\bR$-divisor $B\geq 0$ on $X$ such that $K_{\CF}+B$ is $\mathbb R$-Cartier. If $\CF$ is algebraically integrable, then we say that $(X,\CF,B)$ is algebraically integrable. If $X$ is $\bQ$-factorial, then we say that $(X,\CF,B)$ is $\bQ$-factorial.
\end{defn}

\begin{defn}[Singularities]\label{defn: foliation singularity}
Let $(X,\CF,B)$ be a foliated triple. For any prime divisor $E$ over $X$, let $f: Y\rightarrow X$ be a projective birational morphism such that $E$ is on $Y$, and suppose that
$$K_{\CF_Y}+B_Y:=f^*(K_\CF+B)$$
where $\CF_Y:=f^{-1}\CF$. We define $a(E,\CF,B):=-\mu_EB_Y$ to be the \emph{discrepancy} of $E$ with respect to $(X,\CF,B)$.
We say that $(X,\CF,B)$ is \emph{lc} (resp. \emph{klt}) if $a(E,\CF,B)\geq -1$ (resp. $>-1$) for any prime divisor $E$ over $X$. We say that $(X,\CF,B)$ is \emph{canonical} if $a(E,\CF,B)\geq 0$ for any prime divisor $E$ that is exceptional$/X$.
\textbf{Note that our definition of klt aligns with} \cite{LMX24} \textbf{and is different from
other references.}
 
 An \emph{lc place} of $(X,\CF,B)$ is a prime divisor $E$ over $X$ such that $a(E,\CF,B)=-\epsilon_{\CF}(E)$. An \emph{lc centre} of $(X,\CF,B)$ is the centre $V$ of an lc place of $(X,\CF,B)$ on $X$.
\end{defn}

%%%%%%%%%%%%%%
\subsection{Special algebraically integrable foliations}

\begin{defn}[{Foliated log smooth, \,cf. \cite[\S 3.2]{ACSS21}}] Let $(X,\CF,B)$ be a foliated triple. We say that $(X,\CF,B)$ is \emph{foliated log smooth} if there exists a contraction $f:X\to Z$ satisfying the following conditions
\begin{enumerate}
    \item $\CF$ is induced by $f$,
    \item $(X,\Sigma_X)$ is a quasi-smooth toroidal couple for some divisor $\Sigma_X$ such that $\Supp B\subset\Sigma_X$,
    \item there exists a log smooth couple $(Z,\Sigma_Z)$ such that $$f:(X,\Sigma_X)\to (Z,\Sigma_Z)$$ is an equi-dimensional toroidal contraction.
\end{enumerate}
We say that $f:(X,\Sigma_X)\to (Z,\Sigma_Z)$ is \emph{associated with} $(X,\CF,B)$.  
\end{defn}

\begin{rem}\label{fol log sm is lc}
 If $(X,\CF,B)$ is a foliated log smooth triple, then by \cite[Lemma 3.1]{ACSS21}, $(X,\CF,B^{\CF})$ is lc.
\end{rem}

\begin{defn}\label{def: fol log res}
Let $(X,\CF,B)$ be an algebraically integrable foliated triple. A \emph{foliated log resolution} of $(X,\CF,B)$ is a projective birational morphism $h:X'\to X$ such that $$(X',\CF':=h^{-1}\CF,B':=h_*^{-1}B+\Exc(h))$$
is foliated log smooth. The existence of a foliated log resolution for any such $(X,\CF,B)$ is guaranteed by \cite[Lemma 6.2.4]{CHLX23}.
\end{defn}

%%We refer to \cite[Definition 3.5]{ACSS21} or \cite[Definition 7.2.2]{CHLX23} for the definition of Property $(*)$ foliations. It will be used in the following definition, but not elsewhere.

%%\begin{defn}[ACSS, {cf. \cite[Definitions 5.4.2, 7.2.2, 7.2.3]{CHLX23}}]\label{defn: ACSS f-triple}Let $(X,\CF,B)$ be an lc foliated triple, $G\geq 0$ a reduced divisor on $X$, and $f: X\rightarrow Z$ a contraction. We say that $(X,\CF,B;G)/Z$ is \emph{ACSS} if the following conditions hold:\begin{enumerate}    \item $(X,\CF,B;G)/Z$ satisfies Property $(*)$.\item $f$ is equidimensional.\item There exists an $\bR$-Cartier $\bR$-divisor $D\geq 0$ on $X$, such that  $\Supp\{B\}\subset\Supp D$, and for any reduced divisor $\Sigma\geq f(G)$ such that $(Z,\Sigma)$ is log smooth, $$(X,B+D+G+f^*(\Sigma-f(G)))$$ is qdlt (cf. \cite[Definition 35]{dFKX17}).\item For any lc center of $(X,\CF,B)$ with generic point $\eta$, over a neighborhood of $\eta$, begin{enumerate}\item $\eta$ is the generic point of an lc center of $(X,\CF,\lfloor B\rfloor)$, and \item $f: (X,B+G)\rightarrow (Z,f(G))$ is a toroidal morphism,\end{enumerate}\end{enumerate}
%% If $(X,\CF,B;G)/Z$ is ACSS, then we say that $(X,\CF,B)/Z$ and $(X,\CF,B)$ are ACSS.\end{defn}

We refer to \cite[Definition 5.4.2, 7.2.2, 7.2.3]{CHLX23} for the technical definition of ACSS foliated triples. We only need the following Lemma \ref{ACSS model exist} regarding the existence of ACSS model.
\begin{defn}[ACSS model, {cf. \cite[Definition 7.4.1]{CHLX23}}]
    Let $(X,\CF,B)$ be an lc algebraically integrable foliated triple. An \emph{ACSS modification} of $(X,\CF,B)$ is a birational morphism $h:X'\to X$ such that
    \begin{enumerate}
        \item $(X',\CF':=h^{-1}\CF,B':=h_*^{-1}B+\Exc(h)^{\CF'}$ is ACSS,
        \item $X'$ is klt, and
        \item for any $h$-exceptional prime divisor $E$, we have $a(E,\CF,B)=-\epsilon_{\CF}(E).$
    \end{enumerate}
We say that $(X',\CF',B')$ is an ACSS model of $(X,\CF,B)$. 
\end{defn}

\begin{lem}[{cf. \cite[Theorem 2.5.1]{CHLX23} and \cite[Theorem 3.10]{ACSS21}}]\label{ACSS model exist}
    Let $(X, \CF,B)$ be an lc algebraically integrable foliated triple. Then $(X,\CF,B)$ has a $\bQ$-factorial ACSS model.
\end{lem}

%%%%%%%%%%
\section{DCC of volumes of triples with a fixed birational model}\label{section: dcc for fixed model}
In this section, we treat a special case of Theorem \ref{main thm last} for foliated triples that are birational to a fixed model (cf. Proposition \ref{dcc of fixed model}). The proof here generally follows \cite[Proposition 4.2]{Bir21}, with some necessary modifications. For the notion of b-divisors we refer to \cite{BZ16}.

\begin{lem}[{cf. \cite[Lemma 3.18]{CS25}}]\label{subadjunction}Assume that $(X,\mathcal{F},B)$ is a projective lc foliated log smooth triple, associated with $f:(X,\Sigma_X)\rightarrow (Z,\Sigma_Z)$. Let $V$ be an lc centre of $(X,\CF,B)$ which is horizontal$/Z$. Then the adjunction formula $K_{\CF_V}+B'_V=(K_{\CF}+B)|_V$ holds, and $(V,\CF_V,B'_V)$ is lc, where $\CF_V$ is the restricted foliation of $\CF$ on $V$. Moreover, if $V$  is minimal with respect to inclusion among the lc centres horizontal$/Z$, then $(V,\CF_V,B'_V)$ is klt.
\end{lem}
\begin{proof}
By assumption, there exists a divisor $E$ over $X$ such that $a(X,\CF,B)=-1$ and the centre of $E$ on $X$ is $V$. Since $V$ is horizontal$/Z$, we have $a(E,X,B)=a(X,\CF,B)=-1$. Therefore $E$ is toroidal with respect to $(X,\Sigma_X)$ (cf. Notation \ref{toroidal extraction}) and $V$ is a stratum of $\Sigma^{\CF}_X$. 
Let $\CF_V$ be the restricted foliation of $\CF$ on $V$ which is induced by $V\to Z$. By inductively using \cite[Theorem~6.6.1]{CHLX23}, we obtain the adjunction formula $K_{\CF_V}+B'_V=(K_{\CF}+B)|_V$ such that $(V,\CF_V,B'_V)$ is lc.
 
Suppose $V$ is minimal among the lc centres of $(X,\CF,B)$ horizontal$/Z$. Since all components of $B$ are horizontal$/Z$, $V$ is a minimal lc centre of $(X,B)$. By considering the adjunction formula $K_V+B_V=(K_X+B)|_V$ as defined in Notation \ref{strict toroidal}, we see that $(V,B_V)$ is klt. By inductively applying \cite[Lemma 3.18]{CS25}, we have $B'_V=B^{\CF_V}_V$. Therefore, $a(D,\CF_V,B_V')=a(D,V,B'_V)\ge a(D,V,B_V)>-1$ for every prime divisor $D$ over $V$ with $\epsilon_{\CF_V}(D)=1$. Consequently, $(V,\CF_V,B'_V)$ is klt.
\end{proof}

\begin{prop}\label{dcc of fixed model}
Let $I \subset \mathbb{R}^{\geq0}$ be a DCC set. Assume that $(Y,\mathcal{G},\Delta)$ is a projective foliated log smooth triple associated with $f:(Y,\Sigma_Y)\rightarrow (Z,\Sigma_Z)$ such that $\Delta=\Sigma^{\CG}_Y$. Let $\CS$ be the set of foliated triples $(X,\mathcal{F},B)$ equipped with a birational morphism $\phi:X\rightarrow Y$ satisfying
\begin{enumerate}
    \item $(X,\CF,B)$ is projective lc,
    \item the coefficients of B are in $I$,
    \item $\phi_{*}B\leq \Delta$, and $\phi_*\CF=\CG.$
\end{enumerate}
Then the set \begin{equation*}
	\{\vol(X,K_{\CF}+B)\,|\,(X,\mathcal{F},B)\in \CS\}
\end{equation*}
satisfies the DCC.
\end{prop}
\begin{proof}
\noindent\textbf{Step 1}. By assumption and Remark \ref{fol log sm is lc}, $(Y,\CG,\Delta)$ is lc. Since $I$ is DCC, its closure $\Bar{I}$ is also DCC. By replacing $I$ with $\Bar{I}\cup\{1\}$, we may assume that $I$ is closed and $1\in I$. If the proposition does not hold, then there is a sequence $(X_i,\CF_i,B_i)$ with $\phi_i:X_i\rightarrow Y$ in $S$ such that the volumes $v_i=\vol(X_i,K_{\CF_i}+B_i)$ form a strictly decreasing sequence of positive real numbers.  Let $v=\lim_{i\to\infty}v_i$.

\medskip
\noindent\textbf{Step 2}. By Lemma \ref{ACSS model exist}, we can replace $(X_i,\CF_i,B_i)$ and assume that it is $\bQ$-factorial ACSS. By \cite[Theorem 9.4.1]{CHLX23}, after running a $(K_{\CF_i}+B_i)$-MMP$/Y$ with scaling of an ample divisor, we reach a model on which the pushdown of $K_{\CF_i}+B_i$ is a limit of movable$/Y$ $\bR$-divisors. Replace $X_i$ with that model. By the general negativity lemma \cite[Lemma 3.3]{Bir12}, we have$$K_{\CF_i}+B_i+G_i=\phi_i^*(K_{\CG}+C_i)$$ for some $G_i\geq0$ where $C_i=\phi_{i*}B_i$. Since $C_i\leq\Delta$, the discrepancies satisfy $$a(D,\CG,\Delta)\leq a(D,\CG,C_i)=a(D,\CF_i,B_i+G_i)<0.$$ 
for every component $D$ of $B_i+G_i$. Since $(Y,\CG,\Delta)$ is lc, it follows that $\epsilon_{\CF_i}(D)=1$. Thus, $a(D,Y,\Sigma_Y)\le a(D,Y,\Delta)=a(D,\CG,\Delta)<0$. Therefore $D$ is toroidal with respect to $(Y,\Sigma_Y)$ and its centre on $Y$ is a stratum of $\Delta=\Sigma^{\CG}_Y$.

\medskip
\noindent\textbf{Step 3}. For each $i$, let $\mathbf{M}_i$ be the b-divisor as follows: its trace on $X_i$ is $\mathbf{M}_{i,{X_i}}=B_i$. For any prime divisor $E$ exceptional$/X$, its coefficient is 1 if $\epsilon_{\CF_i}(E)=1$, and is 0 otherwise. The set of all prime divisors $D$ over $Y$ such that 
$0<\mu_D\mathbf{M}_i<1$ for some $i$ is countable, as for each fixed $i$ there are only finitely many $D$ 
with $0<\mu_D\mathbf{M}_i<1$. Therefore, since $I$ is a DCC set, by using a standard diagonalization argument, we may assume that for each prime divisor $D/Y$, the sequence $\mu_D \mathbf{M}_i$ is an increasing sequence for $i\gg 0$. 
Thus, we can define a limiting b-divisor $\mathbf{C}$ by setting $\mu_D\mathbf{C}=\lim_i \mu_D \mathbf{M}_i$
for each $D$. By \textbf{Step 1}, the coefficients of $\mathbf{C}$ are in $I$.

\medskip
\noindent\textbf{Step 4}. Note that $(Y,\Sigma_Y)$ is a quasi-smooth toroidal couple. Let $D$ be a toroidal divisor with respect to $(Y,\Sigma_Y)$ whose centre of $D$ on $Y$ is a stratum of $\Delta$. By Notation \ref{toroidal extraction}, let $p:Y'\to Y$ be the toroidal extraction of $D$. Define $\Sigma_{Y'}$ and $\Delta'$ by $K_{Y'}+\Sigma_{Y'}=p^*(K_Y+\Sigma_Y)$ and $K_{Y'}+\Delta'=p^*(K_Y+\Delta)$. By Lemma \ref{extract hor wss}, $(Y',\CG',\Delta')$ is also foliated log smooth, associated with $f':(Y',\Sigma_{Y'})\to (Z,\Sigma_Z)$, where $\CG'=p^{-1}(\CG)$.

For each $i$, let $W$ be a foliated log resolution of $(X_i,\CF_i,B_i)$ such that the induced map $W\dasharrow Y'$ is a morphism. Consider the new foliated triple $(W,\CF_W,\mathbf{M}_{i,W})$ where $\CF_W$ is the pullback of $\CF_i$ on $W$. This triple is lc, $\bQ$-factorial ACSS, and it satisfies $\vol(W,K_{\CF_W}+\mathbf{M}_{i,W})=\vol(X_i,K_{\CF_i}+B_i)$. After running a $(K_{\CF_W}+\mathbf{M}_{i,W})$-MMP$/Y'$ as in \textbf{Step 2}, we obtain a model $(X_i',\CF_i',B_i')$ over $Y'$ such that $K_{\CF'_i}+B'_i$ is a limit of movable$/Y'$ $\bR$-divisors. Here, $B_i'$ denotes the pushdown of $\mathbf{M}_{i,W}$, 
that is, $B_i'=\mathbf{M}_{i,X_i'}$. By the generalized negativity lemma \cite[Lemma 3.3]{Bir12}, the map $X_i\bir X_i'$ does not contract any divisor. Thus, the b-divisors $\mathbf{M}_i$ and their limit $\mathbf{C}$ are preserved. Let $\phi_i':X_i'\to Y'$ be the natural morphism, and note that $$v_i=\vol(X'_i,K_{\CF_i'}+B_i')$$

In the subsequent steps, we will replace $(Y,\CG,\Delta)$ with $(Y',\CG',\Delta')$ and replace $(X_i,\CF_i,B_i)$ with $(X_i',\CF_i',B_i')$ whenever necessary.

\medskip
\noindent\textbf{Step 5}. Let $C:=\mathbf{C}_Y=\lim_i C_i$. Let $\mathcal{D}_{<}(\CG,C)$ be the set of prime divisors $D$ exceptional$/Y$, such that
$$
 \ \ \ \mu_D\mathbf{C}< -a(D,\CG,C). \ \ \ \ \ \ \ \ \ 
$$  

Note that each $D\in \mathcal{D}_{<}(\CG,C)$ is $\CG$-non-invariant, toroidal with respect to $(Y,\Sigma_Y)$, and the centre of $D$ is a stratum of $\Delta$. Let $p:Y'\to Y$ be as in \textbf{Step 4} and let $C':=\mathbf{C}_{Y'}$. Then we have $K_{\CG'}+C'< p^*(K_\CG+C)$. In particular,$$
\mathcal{D}_{<}(\CG',C')\subsetneq\mathcal{D}_{<}(\CG,C)
$$ 
as $D$ does not belong to the former set. 

\medskip
\noindent\textbf{Step 6}. We associate a weight $w=(p,r,l,d)$ to $(Y,\CG,C)$ as follows. 
We first define the weight $w_V$ for each lc centre $V$ of $(Y,\CG,C)$ which is horizontal$/Z$. Note that each such $V$ is a stratum of $\Delta$. Let $r$ be the codimension of $V$. We have an adjunction formula $K_{\CG_V}+C_V=(K_{\CG}+C)|_V$ by Lemma \ref{subadjunction}, such that $(V,\CG_V,C_V)$ is lc,  where $\CG_V$ is the restricted foliation on $V$. Let $l$ be the number of prime  exceptional divisors $S$ over $V$ mapping into the klt locus of $(V,\CG_V,C_V)$ and satisfying $a(S,\CG_V,C_V)<0$. Then $l<+\infty$. Note that each such $S$ is $\CG_V$-non-invariant, and the klt locus of $(V,\CG_V,C_V)$ is always non-empty. Let $d$ be the sum of coefficients of $C_V$. Put $w_V=(r,l,d)$.

Now we define the weight $w$ of $(Y,\CG,C)$. Let $\mathcal{V}(\CG,C)$ be the set of lc
centres $V$ of $(Y,\CG,C)$ which are horizontal$/Z$ and intersect the centre of some $D\in \mathcal{D}_{<}(\CG,C)$. Let $p$ be the number of elements of $\mathcal{V}(\CG,C)$. If $p=0$, then let $r=0$, let $l$ be the number of elements of $\mathcal{D}_{<}(\CG,C)$, which is finite in this case, and let $d$ be the sum of the coefficients of $C$. If $p>0$, then choose $V\in \mathcal{V}(\CG,C)$ with maximal weight $w_V=(r,l,d)$ with respect to the lexicographic order on 3-tuples. In any case, let $w:=(p,r,l,d)$.

Note that $p,r,l\in\mathbb{N} \cup \{0\}$ while $d$ belongs to a DCC set $I'$ depending only on $I$ because the coefficients of $C_V$ belong to a DCC set $J$ depending only on $I$ (cf. \cite[Theorem~2.4.3]{CHLX23}). We consider the  lexicographic order on the weights $w$.

\medskip
\noindent\textbf{Step 7}. Assume that $(p,r,l,d)$ is the weight of $(Y,\CG,C)$. Suppose $(p,r,l)\neq (0,0,0)$. If $p=0$, then $r=0$ but $l\neq 0$, so there is $D\in \mathcal{D}_{<}(Y,C)$ with its centre $L$ 
contained in the klt locus of $(Y,C)$. 
If $p>0$, then $r>0$ and there is $D\in \mathcal{D}_{<}(\CG,C)$ with its centre $L$ intersecting some 
element of $\mathcal{V}(\CG,C)$.
In either case, let $p\colon Y'\to Y$ and $C'$ be as in \textbf{Step 5} constructed for $D$. 
Then each element of $\mathcal{V}(\CG',C')$ maps birationally onto an element of $\mathcal{V}(\CG,C)$.

Let $w'=(p',r',l',d')$ be the weight of $(Y,\CG',C')$ defined similarly as in \textbf{Step 6}. Then $p'\le p$ by the previous paragraph.
If $p=0$, then $p'=r'=0$ but $l'<l$, so $w'<w$. If $p'<p$, then again $w'<w$. 
This situation could happen. For example, if there exists an lc centre $V$ of $(Y,\CG,C)$ which is horizontal$/Z$ and contained in $L$, then $p'<p$. 
In these cases, we replace $(Y,\CG,\Delta)$ with $(Y',\CG',\Delta')$ as in \textbf{Step 4}. 
This decreases the weight of $(Y,\CG,C)$. 

Now assume that $p'=p>0$, in particular, no horizontal$/Z$ lc centre of $(Y,\CG,C)$ is contained in $L$. Assume that $V\in \mathcal{V}(\CG,C)$ is minimal with respect to inclusion 
among the horizontal$/Z$ lc centres intersecting $L$. Let 
$V'\subset Y'$ be its birational transform. Then $V'$ is a horizontal$/Z$ lc centre of $(Y',\CG',C')$. Assume $w_V=(s,m,e)$ and $w_{V'}=(s',m',e')$. 
\begin{claim}\label{claim: weight inequality}
   $w_{V'}<w_V$. 
\end{claim}
\begin{proof}
Clearly $s'=s$. Define
$$
K_{\CG_V}+C_{V}=(K_{\CG}+C)|_{V}\,\, \text{and}\,\, K_{\CG_{V'}'}+C'_{V'}=(K_{\CG'}+C')|_{V'}
$$ 
by adjunction as in Lemma \ref{subadjunction}. Then $(V,\CG_V,C_V)$ is klt near $L\cap V$ by the minimality of $V$ and Lemma \ref{subadjunction}. By construction,  $K_{\CG'}+C'\lneq \pi^*(K_\CG+C)$. Thus, $K_{\CG_{V'}'}+C_{V'}'\lneq \rho^*(K_{\CG_V}+C_{V})$ where $\rho$ denotes the morphism $V'\to V$. The equality holds over the non-klt locus of $(V,\CG_V,C_V)$ as this non-klt locus is disjoint from $L$. Thus $m'\le m$. If $\rho$ contracts a divisor, the centre of this divisor on $V$ is an irreducible component of $L\cap V$, which is a stratum of $\Delta$. Thus, this exceptional divisor is $\CG'_{V'}$-non-invariant, and consequently $m'<m$. If $\rho$ does not contract any divisor, then $\rho$ is small. We deduce that $e'<e$. Thus the claim $w_{V'}<w_{V}$ follows.
\end{proof}
\noindent\emph{Proof of Proposition \ref{dcc of fixed model} continued.} Assume $U'\in\mathcal{V}(\CG',C')$ and let $U$ be its image on $Y$. Then $U\in\mathcal{V}(\CG,C)$. If $U$ does not intersect $L$, then $w_{U'}=w_U\le w$. If $U$ intersects $L$, then $w_{U'}<w$. Indeed, if $U$ is minimal among the horizontal$/Z$ lc centres intersecting $L$, then $w_{U'}<w_U\le w$ by Claim \ref{claim: weight inequality}. If $U$ is not minimal, then there is $T\in\mathcal{V}(\CG,C)$ intersecting $L$ and $T\subsetneq U$, so $w_{U'}<w_T\le w$.

Now we replace $(Y,\CG,\Delta)$ with $(Y',\CG',\Delta')$ as in \textbf{Step 4} and repeat this step. The above arguments show that eventually we can decrease the weight of $(Y,\CG,\Delta)$ by choosing $D$ appropriately in each step. Since the weight is discrete, after finitely many steps we arrive at the case $p=r=l=0$.

\medskip
\noindent\textbf{Step 8}. From now on we assume $w=(0,0,0,d)$. Equivalently, $\mathcal{D}_{<}(\CG,C)=\emptyset$. If there exists a divisor $D$ exceptional$/Y$ with $a(D,\CG,C)<0$, such that $\mu_D\mathbf{C}= -a(D,\CG,C)$ and its centre on $Y$ is not contained in $\Supp C$, then we replace $(Y,\CG,C)$ with $(Y',\CG',C')$, where the latter is constructed for $D$ as in \textbf{Step 4}. The condition $\mathcal{D}_{<}(\CG,C)=\emptyset$ is preserved. 
Since $(Y,\CG,0)$ is klt outside $\Supp C$, by repeating this finitely many times we can assume that there is no such $D$. 
\begin{claim}\label{claim: inequality}
Suppose $t\in(0,1)$. Then for $i\gg 0$, we have $K_{\CF_i}+B_i\ge \phi_i^*(K_{\CG}+tC)$.   
\end{claim}
\begin{proof}Since $(Y,\CG,tC)$ is klt, there are finitely many prime divisors $D$ over $Y$ with $a(D,\CG,tC)<0$. Note that each such $D$ is toroidal with respect to $(Y,\Sigma_Y)$ and its centre on $Y$ is a stratum of $\Delta$.
It is enough to check the inequality in Claim \ref{claim: inequality} for these $D$. Furthermore, for such $D$ we can assume that $\mu_D B_i$ is sufficiently close to $\mu_D \mathbf{C}_{X_i}$. 

If $D$ is not exceptional$/Y$, then the claim follows from $t<1$ and $i\gg 0$. If $D$ is exceptional$/Y$, then since $\mathcal{D}_{<}(\CG,C)=\emptyset$,
 writing $K_{\CF_i}+A_i=\phi_i^*(K_\CG+C)$, we have $\mu_D \mathbf{C}_{X_i}\ge \mu_DA_i$. Thus,  
$\mu_D \mathbf{C}_{X_i}=\mu_DA_i$ by \textbf{Step 2} as $C_i={\phi_i}_*B_i\le C$. Therefore $D$ satisfies $\mu_D\mathbf{C}=-a(D,\CG,C)$. By the argument in the first paragraph of this step, the centre of $D$ is contained in $\Supp C$.   
The claim then again follows from $t<1$.  
\end{proof}
\noindent\emph{Proof of Proposition \ref{dcc of fixed model} continued.}
By Claim \ref{claim: inequality}, for $i\gg 0$ we get $$v_i=\vol(X_i,K_{\CF_i}+B_i)\geq \vol(Y,K_\CG+tC).$$
Taking the limit as $i\to\infty$, we have $v\geq\vol(Y,K_\CG+tC)$. Now taking the limit as $t\to 1^{-}$, we have  $$v\geq\vol(Y,K_\CG+C)\geq\vol(X_i,K_{\CF_i}+B_i)=v_i.$$
Consequently, $v=v_i$ for $i\gg0$, which is a contradiction to the assumption that $v_i$ is strictly decreasing.
\end{proof}

%%%%%%%%%%
\section{Deformation invariance of relative log canonical volumes}\label{section4}
In this section, we prove Theorem \ref{inv vol semistable} by establishing Theorem \ref{inv vol}, where the semistability condition is replaced by weak semistability. Note that we do not establish the deformation invariance of $h^0(X_t,m(K_{X_t/Z_t}+B_t))$. Instead, we prove the invariance for an auxiliary adjoint structure (cf. Proposition \ref{inv of adj plurigenera}). This construction allows us to apply the Kawamata--Viehweg vanishing theorem, which in turn is sufficient to deduce the invariance of volumes by taking the limit. 

We start with the following lemma.
\begin{lem}\label{lem: fibre of toroidal}
    Let $(X,\Sigma_X)$ be a quasi-smooth toroidal couple and $T$ a smooth variety. Suppose that $\pi:(X,\Sigma_X)\to T$ is a toroidal contraction. Then for each closed $t\in T$, the fibre $(X_t,\Sigma_{X_t})$ is a quasi-smooth toroidal couple.
\end{lem}
\begin{proof}
    Note that the base $T$ is smooth with trivial toroidal structure. Since $\pi$ is a contraction, by \cite[Lemma 6.2]{Qu25}, each $X_t$ is a normal variety, and $(U_{X_t}\subset{X_t})$ defines a toroidal embedding, where $U_{X_t}:=X_t\backslash \Sigma_{X_t}$. By Lemma \ref{lem: restriction of toroidal}, $\Sigma_{X_t}$ is a reduced divisor on $X_t$. Thus, $(X_t,\Sigma_{X_t})$ is a toroidal couple. 
    
    To prove the quasi-smoothness, we consider the local toric models. By the structure of toric morphisms (see, e.g., \cite[Page 133]{Cox11}, the local model of the fibre corresponds to a fibre of the toric morphism between the local toric models. Since $(X,\Sigma_X)$ is quasi-smooth, its local models are simplicial. This implies that the local models of $(X_t,\Sigma_{X_t})$ are also simplicial, ensuring that $(X_t,\Sigma_{X_t})$ is quasi-smooth.
\end{proof}

\begin{defn}[Family of weak semistable morphisms]\label{weak ss family}
	Let $f:X\to Z$ be a contraction over a smooth variety $T$, and let $\Sigma_X$ and $\Sigma_Z$ be divisors on $X$ and $Z$, respectively. We say that $f:(X,\Sigma_X)\to (Z,\Sigma_Z)/T$ is a \emph{family of weak semistable morphisms} if 
 \begin{enumerate}
     \item $f:(X,\Sigma_X)\to(Z,\Sigma_Z)$ is weak semistable (cf. Definition \ref{def: wss}), and
     \item  the structure morphism $Z\to T$ is a contraction and $(Z,\Sigma_Z)$ is log smooth over $T$.
 \end{enumerate}
\end{defn}

\begin{rem}
	Suppose $f:(X,\Sigma_X)\to (Z,\Sigma_Z)/T$ is a family of weak semistable morphisms. Then the induced morphism $\pi:(X,\Sigma_X)\to T$ is toroidal, as it is the composition of toroidal morphisms. By Lemma \ref{lem: fibre of toroidal}, each closed fibre $(X_t,\Sigma_{X_t})$ is a quasi-smooth toroidal couple. Since $(Z,\Sigma_Z)$ is log smooth over $T$ and $Z\to T$ is a contraction, each $(Z_t,\Sigma_{Z_t})$ is a log smooth couple. The induced fibrewise morphism $f_t:(X_t,\Sigma_{X_t})\to (Z_t,\Sigma_{Z_t})$ is clearly flat with reduced fibres. Furthermore, by \cite[Theorem 3.11]{BQ24}, $f_t$ is a toroidal morphism. Therefore, $f_t$ is a weak semistable morphism.
\end{rem}

\begin{lem}\label{extract family}
    Let $\pi:(X,\Sigma_X)\to T$ be a toroidal contraction from a quasi-smooth toroidal couple $(X,\Sigma_X)$ to a smooth curve $T$. Let $E$ be a toroidal divisor with respect to $(X,\Sigma_X)$, and let $p:Y\to X$ be the toroidal extraction of $E$. Let $\Sigma_Y$ be the divisor defined by $K_Y+\Sigma_Y=p^*(K_X+\Sigma_X)$. Suppose that each stratum of $\Sigma_X$ has irreducible fibres over $T$. 
    
    Then for each closed point $t\in T$, the fibrewise morphism $p_t:(Y_t,\Sigma_{Y_t})\to (X_t,\Sigma_{X_t})$ is the toroidal extraction of the toroidal divisor $E_t$.
   \end{lem}
    \begin{proof}
      By assumption $(Y,\Sigma_Y)$ is a quasi-smooth toroidal couple, and the induced morphism $\pi_Y:(Y,\Sigma_Y)\to T$ is a toroidal contraction. By Lemma \ref{lem: fibre of toroidal}, each fibre $(Y_t,\Sigma_{Y_t})$ is also quasi-smooth toroidal. Let $S$ be the centre of $E$ on $X$, then $S$ is a stratum of $\Sigma_X$. By assumption the intersection $S_t:=S\cap X_t$ is irreducible and is a stratum of $\Sigma_{X_t}$. Consequently, the exceptional locus of the fibrewise morphism $p_t:Y_t\to X_t$ is exactly $E_t$, a single prime divisor. Moreover, since $-E$ is $p$-ample, $-E_t$ is $p_t$-ample. Therefore, $p_t$ is the toroidal extraction of $E_t$.
    \end{proof}

%%%%%%%(invariance of volume)
\begin{thm}\label{inv vol}
	Let $f:(X,\Sigma_X)\to (Z,\Sigma_Z)/T$ be a family of weak semistable morphisms. Let $B$ be an $\bR$-divisor on $X$ such that $0\le B\le \Sigma^h_X$, where $\Sigma^h_X$ denotes the horizontal$/Z$ part of $\Sigma_X$. 
 
 Then the relative log canonical volume $\vol(X_t,K_{X_t/Z_t}+B_t)$ is independent of $t\in T$.
\end{thm}
    \begin{proof}
		\noindent\textbf{Step 1}. In this step, we fix notation and make standard reductions. Let $d=\dim X$ and let $\CF$ be the foliation induced by $f$. Since $f$ has reduced fibres, we have $K_\CF=K_{X/Z}$. Clearly,
 both $(X,\CF,B)$ and each fibre $(X_t,\CF_t,B_t)$ are lc. Replacing $T$ with a suitable \'etale cover \cite[Claim 4.38.1]{Kol13}, we may assume that each stratum of $\Sigma^h_X$ has irreducible fibres over $T$. Replacing $T$ with the intersection of general hyperplane sections and shrinking $T$, we may assume that $T$ is a smooth affine curve. 
  
We take a sequence of $\bQ$-divisors $\{B_i\}_{i=1}^{\infty}$ such that $B_i\le B$ and $\lim_{i\to\infty}B_i=B$. It suffices to prove the statement for each $B_i$, so we may assume $B$ is a $\bQ$-divisor.
By the continuity of volumes, we have $$\lim_{\epsilon\to 0}\vol\left(X_t,K_{\CF_t}+\left(1-\epsilon\right)B_t\right)=\vol(X_t,K_{\CF_t}+B_t).$$ Therefore we may assume that $\lfloor B\rfloor=0$. This implies that both $(X,B)$ and each $(X_t,B_t)$ are klt. Finally, we may assume that there exists a closed point $s\in T$ such that $$\vol(X_s,K_{\CF_s}+B_s)>0,$$ otherwise there is nothing to prove. 
        
 \medskip
        \noindent\textbf{Step 2}. In this step, we reduce to the case that $(X,B)$ is canonical and prove that $K_\CF+B$ is pseudo-effective$/T$.
        
        Since $f$ is flat with reduced fibres, we have $\Sigma^v_X=f^*(\Sigma_Z)$. Consequently, $K_X+\Sigma^h_X$ is a Cartier divisor. Since $(X,B)$ is klt, there are finitely many divisors $E$ exceptional$/X$ with discrepancy $a(E,X,B)<0$, each such divisor is toroidal with respect to $(X,\Sigma_X)$, and its centre on $X$ is a stratum of $\Sigma^h_X$. By applying Lemma \ref{extract hor wss} and Lemma \ref{extract family} finitely many times and replacing $(X,B)$ with the resulting crepant model, we may assume $(X,B)$ is canonical.
        
        To prove that $K_{\CF}+B$ is pseudo-effective$/T$, we are free to shrink $T$ around $s$. Fix a very ample linear series $|L|$ on $Z$. Let $H_1,...,H_{4d+1}$ be $4d+1$ general elements of $|L|$. After shrinking $T$, we may assume that the pair $$\left(Z,\Sigma_Z+\sum_{i=1}^{4d+1}H_i
        \right)$$ is log smooth over $T$. 
        By \cite[Proposition 3.2]{AK00}, we can add $\sum_{i=1}^{4d+1}H_i$ and $f^*(\sum_{i=1}^{4d+1}H_i)$ to the toroidal structure of $Z$ and $X$, respectively, such that $$f:\left(X,\Sigma_X+\sum_{i=1}^{4d+1}f^*H_i\right)\to \left(Z,\Sigma_Z+\sum_{i=1}^{4d+1}H_i\right)$$ remains toroidal and hence weak semistable. By the length of extremal rays, the divisor $K_Z+\frac{1}{2}\sum_{i=1}^{4d+1}H_i$ is ample on $Z$. Then the divisor
         \begin{align*}
K_{X_s}+B_s+H_s=K_{\CF_s}+B_s+f^*_s\left(K_{Z_s}+\frac{1}{2}\sum_{i=1}^{4d+1}H_{i,s}\right)
         \end{align*}
is big, as it is the sum of a big divisor and a nef divisor. Let $H:=\frac{1}{2}f^*(\sum_{i=1}^{4d+1}H_i)$.
\begin{claim}\label{Claim:pseudo-effective}$K_X+B+H$ is pseudo-effective$/T$. 
\end{claim}
\begin{proof}
Passing to an \'etale cover of $T$, we may assume that every stratum of the toroidal couple $$\left(X,\Sigma_X+\sum^{4d+1}_{i=1}f^*H_i\right)$$ has irreducible fibres over $T$ (cf. \cite[Claim 4.38.1]{Kol13}). 
Note that $(X,B+H)$ is klt. By applying Lemma \ref{extract family} finitely many times, we may assume that $(X,B+H)$ is canonical. Note also that $K_{X_s}+B_s+H_s$ remains big. By \cite[Theorem 1.2]{BCHM10}, $(X_s,B_s+H_s)$ has a good minimal model. Moreover, $N_\sigma\left(K_{X_s}+B_s+H_s\right)$, the numerically fixed part of the Nakayama--Zariski decomposition of $K_{X_s}+B_s+H_s$, is a $\bQ$-divisor. We define the divisor
$$\Theta_s:=B_s+H_s-\left(B_s+H_s\right)\wedge N_\sigma\left(K_{X_s}+B_s+H_s\right).$$
Then $\Theta_s\in \bQ$, $K_{X_s}+\Theta_s$ remains big and no component of $\Theta_s$ is contained in $\mathbf{B}_{-}\left(K_{X_s}+\Theta_s\right)$. Note that $(X_s,\Theta_s)$ still has a good minimal model. Pick the unique divisor $\Theta$ on $X$ such that $0\le\Theta\le B+H$ and $\Theta|_{X_s}=\Theta_s$. Now by \cite[Lemma 3.2]{HMX2}, $K_{X}+\Theta$ is pseudo-effective$/T$. Thus, $K_X+B+H$ is also pseudo-effective$/T$. The claim follows.
\end{proof}
\noindent\emph{Proof of Theorem \ref{inv vol} continued.} By Claim \ref{Claim:pseudo-effective} and \cite[Theorem 5.2 or Theorem 5.3]{ACSS21}, $K_{\CF}+B$ is pseudo-effective$/T$ (see also \cite[Theorem 5.1]{CHLMSSX25}).

\medskip
\noindent\textbf{Step 3}. In this step, we finish the proof by using Proposition \ref{inv of adj plurigenera} below. 
By a slight abuse of notation, we may shrink $T$ and assume that there exist $4d+1$ ample divisors $H_1,...,H_{4d+1}$ on $Z$ belonging to a fixed very ample linear series such that $(Z,\Sigma_Z+\sum_{i=1}^{4d+1}H_i)$ is log smooth over $T$. As before, let $H:=\frac{1}{2}f^*(\sum_{i=1}^{4d+1}H_i)$. We also fix a general very ample divisor $A$ on $X$ and a positive integer $k>1$ such that $(X,B+\frac{1}{k}A)$ remains canonical.
		
For any $\delta\in(0,1)$, we define the divisor$$K_{\mathfrak{A}_{\delta}}:=\delta K_{\CF}+(1-\delta)K_X+B+\frac{1}{k}A+(1-\delta)H.$$  
By Proposition \ref{inv of adj plurigenera}, $h^0(X_t,mK_{\mathfrak{A}_{\delta,t}})$ is independent of $t$ for any sufficiently divisible $m$. Thus, the volume $\vol(X_t,K_{\mathfrak{A}_{\delta,t}})$ is also independent of $t$. By taking the limit as $k\to\infty$, we see that $$\vol\Bigl(X_t,\delta K_{\CF_t}+(1-\delta)K_{X_t}+B_t+(1-\delta)H_t\Bigr)$$ is independent of $t$. This holds for all $\delta\in(0,1)$. Finally, taking the limit as $\delta\to 1^{-}$, we conclude that $\vol(X_t,K_{\CF_t}+B_t)$ is independent of $t$.
\end{proof}

%%%%%%%%%(invariance of plurigenera)
\begin{prop}\label{inv of adj plurigenera}
	Let $\delta\in (0,1)$ be a real number and $k>1$ be an integer. Assume that:
 \begin{enumerate}
     \item $f:(X,\Sigma_X)\to (Z,\Sigma_Z)/T$ is a family of weak semistable morphisms with $\dim X=d$,
     \item $B$ is a $\bQ$-divisor on $X$ satisfying $0\le B\le \Sigma^h_X$ and $A$ is a general very ample divisor on $X$ such that $(X,B+\frac{1}{k}A)$ has canonical singularities,
     \item there exist a very ample linear series $|L|$ on $Z$ and divisors $H_1,..,H_{4d+1}\in|L|$ such that $(Z,\Sigma_Z+\sum_{i=1}^{4d+1}H_i)$ is log smooth over $T$, and
     \item  $K_{X/Z}+B$ is pseudo-effective$/T$.
 \end{enumerate}
 Then $h^0(X_t,mK_{\mathfrak{A}_{\delta,t}})$ is independent of $t\in T$ for any sufficiently divisible integer $m>0$, where $$K_{\mathfrak{A}_{\delta}}:=\delta K_{X/Z}+(1-\delta)K_X+B+\frac{1}{k}A+(1-\delta)H,$$ with $H:=\frac{1}{2}f^*(\sum_{i=1}^{4d+1}H_i)$.
\end{prop}
\begin{proof}\textbf{Step 1}. In this step, we fix notation and make standard reductions. Since the statement is local on $T$, we may fix a closed point $0\in T$ and shrink $T$ around $0$. By replacing $T$ with the intersection of general hyperplane sections containing $0$ and shrinking further, we may assume $T$ is a smooth affine curve. Furthermore, by passing to an \'etale cover we may assume that every component of $B$ has irreducible fibres over $T$.
    
Let $\CF$ be the foliation induced by $f$. Since $f$ has reduced fibres, we have $K_\CF=K_{X/Z}$. Define the divisor $$N:=f^*\left(K_Z+\frac{1}{2}\sum_{i=1}^{4d+1}H_i\right).$$ By the length of extremal rays, $K_Z+\frac{1}{2}\sum_{i=1}^{4d+1}H_i$ is ample on $Z$. Thus, $N$ is semiample. We may write:
    \begin{align*}
        K_{\mathfrak{A}_{\delta}}&=\delta\left(K_\CF+B+\frac{1}{k}A\right)+(1-\delta)\left(K_X+B+\frac{1}{k}A+H\right)\\
        &=K_\CF+B+\frac{1}{k}A+\biggl(1-\delta\biggr)N.
    \end{align*}
Note that $$K_\CF+B+\frac{1}{k}A\sim_{\bQ,Z}K_X+B+\frac{1}{k}A+H.$$
The pair $(X,B+\frac{1}{k}A+H)$ is klt. Moreover, $K_X+B+\frac{1}{k}A+H$ is big$/T$. By \cite[Corollary 1.4.2]{BCHM10}, we may run a $(K_X+B+\frac{1}{k}A+H)$-MMP$/T$. This MMP terminates at a good minimal model $X'/T$ of $(X,B+\frac{1}{k}A+H)$. By \cite[Theorem 16.1.4]{CHLX23}, this is also a $(K_\CF+B+\frac{1}{k}A)$-MMP$/T$. Furthermore, $X'$ is the good minimal model$/T$ of the generalized foliated quadruple $(X,\CF,B,\frac{1}{k}\Bar{A})$ (see the following Remark \ref{g-pair}). In particular, $X'$ is the minimal model of $K_{\mathfrak{A}_\delta}$. Hence, $N_\sigma(K_{\mathfrak{A}_\delta})$ is a $\bQ$-divisor. Using the same argument, we deduce that $N_\sigma(K_{\mathfrak{A}_{\delta,0}})$ is also a $\bQ$-divisor.
     
It suffices to show that $|mK_{\mathfrak{A}_{\delta,0}}|=|mK_{\mathfrak{A}_{\delta}}||_{X_0}$.

\medskip
\noindent\textbf{Step 2}. In this step, we make preparations to run an MMP$/T$ in the next step. 
    
We define $\Theta_0:=B_0-B_0\wedge N_{\sigma}(K_{\mathfrak{A}_{\delta,0}})$. Then $\Theta_0$ is a $\bQ$-divisor and $$\Theta_0\wedge N_{\sigma}\left(K_{\CF_0}+\Theta_0+\frac{1}{k}A_0+\biggl(1-\delta\biggr)N_0\right)=0.$$ Here note that: 
 $$ K_{\CF_0}+\Theta_0+\frac{1}{k}A_0+\left(1-\delta\right)N_0=\delta K_{\CF_0}+(1-\delta)K_{X_0}+\Theta_0+\frac{1}{k}A_0+(1-\delta)H_0,$$ and $$K_{\CF_0}+\Theta_0+\frac{1}{k}A_0+N_0=K_{X_0}+\Theta_0+\frac{1}{k}A_0+H_0.$$
Since $N_0$ is semiample, we have $$\mathbf{B}_{-}\left(\delta K_{\CF_0}+\left(1-\delta\right)K_{X_0}+\Theta_0+\frac{1}{k}A_0+\left(1-\delta\right)H_0\right)\subset\mathbf{B}_{-}\left(K_{X_0}+\Theta_0+\frac{1}{k}A_0+H_0\right).$$
    
Note that $(X_0,\Theta_0+\frac{1}{k}A_0)$ is canonical. Each non-canonical centre of this pair is an irreducible component of $\Theta_0+\frac{1}{k}A_0$, and thus is not contained in $\mathbf{B}_{-}\bigl(K_{\CF_0}+\Theta_0+\frac{1}{k}A_0+\left(1-\delta\right)N_0\bigr)$ as $A_0$ is ample. Possibly replacing each $H_i$ up to linear equivalence, we may assume that no non-canonical centre of $(X_0, \Theta_0+\frac{1}{k}A_0+H_0)$ is contained in $\mathbf{B}_{-}\bigl(K_{\CF_0}+\Theta_0+\frac{1}{k}A_0+\left(1-\delta\right)N_0\bigr)$, hence is not contained in $\mathbf{B}_{-}\left(K_{X_0}+\Theta_0+\frac{1}{k}A_0+H_0\right)$.
	
Since every irreducible component of $B$ has irreducible fibres over $T$, we can pick the unique divisor $0\leq\Theta\leq B$ with $\Theta|_{X_0}=\Theta_0$. Note that $K_{\CF_0}+\Theta_0$ is pseudo-effective. By a similar argument as in \textbf{Step 2} of Theorem \ref{inv vol}, $K_{\CF}+\Theta$ is pseudo-effective$/T$. 

    \medskip
    \noindent\textbf{Step 3}. In this step, we use \cite[Lemma 3.1]{HMX2} to run an MMP$/T$ and apply the Kawamata--Viehweg vanishing theorem. Now $(X,\Theta+\frac{1}{k}A+H)$ is klt and $K_X+\Theta+\frac{1}{k}A+H$ is big$/T$. We can run a $(K_X+\Theta+\frac{1}{k}A+H)$-MMP$/T$, which terminates at a good minimal model $Y$ \cite[Corollary 1.4.2]{BCHM10}. Note that $$K_X+\Theta+\frac{1}{k}A+H\sim_{\bQ,Z}K_{\CF}+\Theta+\frac{1}{k}A,$$ by using a similar argument as in \textbf{Step 1} and the results in \cite[Section 9 and 16]{CHLX23}, we see that this MMP is also a $(K_{\CF}+\Theta+\frac{1}{k}A)$-MMP$/T$. 
    
    By \cite[Proposition 3.2]{AK00}, we can add $X_0$ to the toroidal boundary to obtain a new toroidal couple $(X,X_0+\Sigma_X+\sum_{i=1}^{4d+1} H_i)$. Note that the pair $$\left(X,X_0+B+\frac{1}{2}\sum_{i=1}^{4d+1} H_i\right)$$ is plt. Let $\phi:X\dashrightarrow Y$ denote the corresponding birational map. By \cite[Lemma 3.1]{HMX2}, $\phi_0:X_0\dasharrow Y_0$ is a birational contraction. Let $W$ be a common log resolution of $X$ and $Y$, and let $p:W\to X$ and $q:W\to Y$ be the induced morphisms. Then we may write $$K_W+W_0=p^*\left(K_X+X_0+\Theta+\frac{1}{k}A+H\right)+E,$$
	where $W_0$ is the strict transform of $X_0$ and $\lceil E\rceil\geq 0$ is $p$-exceptional.
	We may also write $$p^*\left(K_{\CF}+\Theta+\frac{1}{k}A\right)=q^*\phi_*\left(K_{\CF}+\Theta+\frac{1}{k}A\right)+F_1$$
	and $$p^*\left(K_{X}+\Theta+\frac{1}{k}A+H\right)=q^*\phi_*\left(K_{X}+\Theta+\frac{1}{k}A+H\right)+F_2,$$
	where $F_1$ and $F_2$ are effective $q$-exceptional divisors. For simplicity of notation, we define two divisors $$M_X:=K_{X}+\Theta+\frac{1}{k}A+H,\,\,\, M_{\CF}:=K_{\CF}+\Theta+\frac{1}{k}A.$$ 
 By \cite[Theorem 16.1.4]{CHLX23}, both $\phi_*M_X$ and $\phi_*M_\CF$ are semiample. 
Possibly shrinking $T$, we may assume that $X_0$ is $\bQ$-linearly equivalent to zero. Define
	$$\alpha=mp^*\Bigl(\delta M_{\CF}+(1-\delta)M_X\Bigr)+E-m\delta F_1-\left(m\left(1-\delta\right)-1\right)F_2.$$ 
Set $L=\lceil \alpha\rceil$, $C=\{-\alpha\}$, and $F_0=m\delta F_1+\left(m\left(1-\delta\right)-1\right)F_2$. Then we have:
	\begin{align*}
		L-W_0 &=mp^*\Bigl(\delta M_{\CF}+(1-\delta)M_X\Bigr)+E-F_0+C-W_0\\
		&=p^*M_X+E-W_0+C+p^*\Bigl(m\delta M_{\CF}+(m(1-\delta)-1) M_X\Bigr)-F_0\\
		&\sim_{\bQ}K_W+C+q^*\phi_*\Bigl(m\delta M_{\CF}+(m(1-\delta)-1)M_X\Bigr).
	\end{align*}

The pair $(W,C)$ is klt since it is log smooth and the coefficients of $C$ are strictly less than $1$.  The Kawamata--Viehweg vanishing theorem implies that $$H^1\bigl(W,\mathcal{O}_W(L-W_0)\bigr)=0.$$ 
 Thus, the restriction morphism $$H^0\bigl(W,\CO_{W}(L)\bigr)\to H^0\bigl(W_0,\CO_{W_0}(L|_{W_0})\bigr)$$
	 is surjective. 
     
     \medskip
     \noindent\textbf{Step 4}. In this step, we complete the proof using the result in \textbf{Step 3}. Note that the support of the divisor $L-\Bigl\lfloor mq^*\phi_*\Bigl(\delta M_\CF+(1-\delta)M_X\Bigr)\Bigr\rfloor$ does not contain $W_0$. We have 
	 \begin{align*}
	 	&\quad \quad L-\Bigl\lfloor mq^*\phi_*\Bigl(\delta M_\CF+(1-\delta)M_X\Bigr)\Bigr\rfloor=\lceil \alpha\rceil-\Bigl\lfloor mq^*\phi_*\Bigl(\delta M_\CF+(1-\delta)M_X\Bigr) \Bigr\rfloor \tag{*}\\
	 	&\geq \Bigl\lceil \alpha-mq^*\phi_*\Bigl(\delta M_\CF+\left(1-\delta\right)M_X\Bigr)\Bigr\rceil=\lceil E+F_2\rceil \geq0.
	 \end{align*}
Furthermore, we have the following inclusion 
	\begin{align*}
		|L|&\subset\bigl|mp^*(K_{\mathfrak{A}_{\delta}})+\lceil E-F_0\rceil\bigr|\subset\bigl|mp^*(K_{\mathfrak{A}_{\delta}})+\lceil E\rceil\bigr|=|mK_{\mathfrak{A}_{\delta}}|.\tag{**}
	\end{align*}
 Let $q_0: W_0\to Y_0$ be the restriction of $q$ to $W_0$. We have
\begin{align*}
	|mK_{\mathfrak{A}_{\delta,0}}|
 &=\Bigl|m\Bigl(\delta M_{\CF,0}+\left(1-
	\delta\right) M_{X,0}\Bigr)\Bigr|
 && \text{by definition of}\,\, \Theta_0,\\
 &\subset\Bigl|mq_0^*f_{0*}\Bigl(\delta M_{\CF,0}+\left(1-
	\delta\right) M_{X,0}\Bigr)\Bigr|
 &&\text{since}\,\, f_0 \,\,\text{is a birational contraction,} \\
 &\subset|L_{|W_0}|
 && \text{by (*),}\\
 &=|L||_{W_0} 
 && \text{since}\,\,H^1\bigl(W,\mathcal{O}_W(L-W_0)\bigr)=0,\\
	&\subset|mK_{\mathfrak{A}_{\delta}}||_{X_0}.
 && \text{by (**).}
\end{align*}
Clearly, we have $|mK_{\mathfrak{A}_{\delta}}||_{X_0}\subset|mK_{\mathfrak{A}_{\delta,0}}|$. Therefore, the equality $|mK_{\mathfrak{A}_{\delta,0}}|=|mK_{\mathfrak{A}_{\delta}}||_{X_0}$ holds.
\end{proof}

\begin{proof}[Proof of Theorem \ref{inv vol semistable}]
 This is a special case of Theorem \ref{inv vol}.   
\end{proof}

\begin{rem}\label{g-pair}
    We cannot simply add the ample divisor $\frac{1}{k}A$ to the boundary of $(X,\CF,B)$ because the resulting foliated triple $(X,\CF,B+\frac{1}{k}A)$ may not remain lc. See \cite[Example 3.4]{DLM23}. This issue arises from the failure of Bertini-type theorems. Nevertheless, $(X,\CF,B,\frac{1}{k}\Bar{A})$ is lc as a generalized foliated quadruple. For the notion of generalized foliated quadruples, we follow \cite{CHLX23}. We omit this technical definition here for the reader's convenience.
\end{rem}

%%%%%%%%%%%
\section{Proof of the main theorem}\label{section: main thm}
In this section we prove Theorem \ref{main thm last}. We start with the following definition. For the usual definition of a \emph{bounded family} of foliations, see for example, \cite[Definition 3.31]{CHLMSSX25}.
\begin{defn}[Birationally bounded by algebraically integrable families]\label{bir bdd fol triples}
    Let $\CS$ be a set of foliated triples. We say that $\CS$ is \emph{birationally bounded by algebraically integrable families} if there exist finitely many foliated triples $(X^i,\CF^i,B^i)$, equipped with dominant rational maps $f^i:X^i\dasharrow Z^i$ between normal varieties over base varieties $T^i$ for $i=1,...,N$, such that:
    \begin{enumerate}
        \item $B^i$ is a reduced divisor, and the structural morphisms $X^i\to T^i$ and $Z^i\to T^i$ are projective,
        \item $\CF^i$ is induced by $f^i$ (in particular, it is algebraically integrable), and
        \item for every $(Y,\CG,C)\in \CS$, there exists an index $i\in\{1,...,N\}$, a closed point $t\in T^i$ and a birational map $\phi: Y\dasharrow X^i_t$ such that:
 \begin{enumerate}
    \item the restriction of $f^i$ to the fibre $X^i_t$ induces a well-defined dominant rational map $f^i_t: X^i_t\dasharrow Z^i_t$ between normal varieties,
    \item $(X^i_t,\CF^i_t,B^i_t)$ is a foliated triple, where $\CF^i_t$ is the foliation induced by $f^i_t$ and $B^i_t$ is the restriction of $B^i$ to $X^i_t$,
    \item $\phi_*\CG=\CF^i_t$, and
    \item $\Supp B^i_t$ contains $\Supp \phi_{*}C$ and the support of any $\phi^{-1}$-exceptional divisor.
 \end{enumerate}
\end{enumerate}
\end{defn}

\begin{rem}\label{rem: bir bdd foliations}
    If a set $\CS$ of foliated triples is birationally bounded by algebraically integrable families, then the general leaves of elements in $\CS$ are also birationally bounded as algebraic varieties (cf. \cite[Definition 2.4.1]{HMX1}). However, as discussed in \cite[Example 4.3]{passantino24}, there exists a set $\CS$ of rank one algebraically integrable foliations on surfaces, such that $\CS$ is bounded (cf. \cite[Definition 3.31]{CHLMSSX25}), but the leaves of elements in $\CS$ (which are irreducible curves) have unbounded genus.
\end{rem}

%%%%%%%% proof of the main thm
\begin{proof}[Proof of Theorem \ref{main thm last}]
	The strategy of the proof is similar to that of \cite[Theorem 1.9]{HMX1}.

    \medskip
    \noindent\textbf{Step 1}. In this step we fix notation and make standard reductions. We may assume that $1\in I$. Possibly replacing $\CS$ with a subset, we may assume that
    $\CS$ is birationally bounded by a single algebraically integrable family $f:X\dasharrow Z/T$. Let $(Y,\CG,C)\in \CS$. By assumption, there is a closed point $t\in T$ and a foliated triple $(X_t,\CF_t,B_t)$ equipped with a birational map $\phi: Y\dasharrow X_t$ such that $\phi_*\CG=\CF_t$ and $\Supp \phi_*C\cup\Exc \phi^{-1}\subset\Supp B_t$, where $\CF_t$ is induced by the dominant rational map $f_t:X_t\dasharrow Z_t$, see Definition \ref{bir bdd fol triples}.
    
	Suppose that $p:Y'\to Y$ is a foliated log resolution of $(Y,\CG,C)$ (cf. Definition \ref{def: fol log res}) such that the induced map $Y'\to X_t$ is a morphism. Let $\CG'=p^{-1}\CG$ and $C'=p_*^{-1}C+(\text{Exc p})^{\CG'}$. Then $(Y',\CG',C')\in \CS$ and $$\vol(Y,K_{\CG}+C)=\vol(Y',K_{\CG'}+C').$$ 
    
    Replacing $(Y,\CG, C)$ by $(Y',\CG',C')$, we may assume that $\phi$ is a morphism, and we are free to replace $X$ and $Z$ by higher birational models. We may therefore assume that $f$ is a morphism.
	
	Passing to a stratification of $T$, we may assume that both $f$ and $Z\to T$ are contractions between normal varieties, and the points that parameterize elements in $\CS$ are dense in $T$. Passing to an open set of $T$, we may also assume that the fibres of both $(X,B)\to T$ and $Z\to T$ are log pairs. 
 For any $(Y,\CG,C)\in \CS$, Let $\phi:Y\to X_t$ denote the corresponding birational morphism. 
We have $\phi_*C\leq B_t$. Since $(Y,\CG,C)$ is lc, every component of $\phi_*C$ is horizontal$/Z_t$. Consequently, we may remove the vertical$/Z$ components of $B$ and assume that every component of $B$ is horizontal$/Z$.

    \medskip
	\noindent\textbf{Step 2}. In this step, we reduce to the case where $f:(X,B)\to Z/T$ is a family of weak semistable morphisms (cf. Definition \ref{weak ss family}) using weak semistable reduction \cite{AK00}. 
    
    Replacing $f:(X,B)\to Z$ with its weak semistable model (cf. Theorem \ref{weak semistable reduction}), we may assume that $f:(X,\Sigma_X)\to (Z,\Sigma_Z)$ is weak semistable with respect to the divisors $\Sigma_X$ and $\Sigma_Z$. We may also assume that $B= \Sigma_X^{h}$, the horizontal$/Z$ part of $\Sigma_X$. Passing to a non-empty open subset of $T$ and taking the Stein factorization, we may assume that $Z\to T$ is a contraction to a smooth variety $T$ and $(Z,\Sigma_Z)$ is log smooth over $T$. Thus $(X,B)\to Z/T$ is a family of weak semistable morphisms.
    
    Since the construction above involves taking a finite cover of $Z$, we must correspondingly replace each $(Y,\CG,C)\in \CS$ with a suitable finite cover, such that the coefficients of $C$ still belong to $I$. The existence of such a cover is guaranteed by \cite[Proposition 3.15]{HJLL24} (see also \cite[Proposition 2.2]{CS251} and \cite[Lemmas 3.4 and 4.3]{Dru21}). Note that although the volume $\vol(Y,K_\CG+C)$ is multiplied by the degree of the finite cover, this degree depends only on the initial family $f:X\to Z/T$ and thus belongs to a fixed finite set. Consequently, it suffices to work with the new set of triples consisting of these covers of the elements in the original set. Furthermore, passing to an \'etale cover of $T$ (cf. \cite[Claim 4.38.1]{Kol13}), we may assume that each stratum of $\Sigma^h_X$ has irreducible fibres over $T$.
    
    \medskip
    \noindent\textbf{Step 3}. In this step, we complete the proof using Theorem \ref{inv vol} and Proposition \ref{dcc of fixed model}. Fix a closed point $0\in T$. Let $\CS_0\subset \CS$ be the set of lc projective algebraically integrable foliated triples $(Y,\CG,C)$ such that the coefficients of $C$ belong to $I$, and there exists a birational morphism $\beta:Y\to X_0$ with $\beta_*\CG=\CF_0$ and $\beta_*C\leq B_0$.
	
	Let $(Y,\CG, C)\in \CS$. By assumption, there exists a closed point $t\in T$ and a birational morphism $\phi:Y\to X_t$ such that $\phi_*\CG=\CF_t$. Define $\Phi=\phi_*C$. Note that $\Phi\le B_t=\Sigma_{X_t}^h$. Replacing $(Y,\CG,C)$ with its ACSS model, we may assume that $(Y,\CG,C)$ is $\bQ$-factorial ACSS (cf. Lemma \ref{ACSS model exist}). By \cite[Theorem 9.4.1]{CHLX23}, after running a $(K_\CG+C)$-MMP$/X_t$ with scaling of an ample divisor we reach a model on which the pushdown of $K_\CG+C$ is a limit of movable$/X_t$ $\bR$-divisors. Replacing $(Y,\CG,C)$ with this model and using the general negativity lemma (cf. \cite[Lemma 3.3]{Bir12}), we may assume that $$K_{\CG}+C\leq\phi^*(K_{\CF_t}+\Phi).$$

 If a divisor $E_t$ on $Y$ satisfies $a(E_t,\CF_t,\Phi)<0$, then it is $\CF_t$-non-invariant and toroidal with respect to $(X_t,\Sigma_{X_t})$. Therefore, its centre on $X_t$ is a stratum of $\Sigma^h_{X_t}$. There is a birational morphism $\phi':Y'\to X_t$ that extracts all such divisors $E_t$. Note that $\phi'$ is a sequence of toroidal extractions of strata of $(X_t,\Phi)$ (cf. Notation \ref{toroidal extraction}). This implies that the induced map $\alpha: Y\dasharrow Y'$ is a birational contraction. Set $C'=\alpha_*C$ on $Y'$. Then $(Y',\CG',C')\in \CS$, and $$\vol(Y',K_{\CG'}+C')=\vol(Y,K_{\CG}+C).$$ 
 
 Replacing $(Y,\CG,C)$ with $(Y',\CG',C')$, we may assume that $\phi$ is a sequence of toroidal extractions of strata of $\Sigma^h_{X_t}$.
Note that all the divisors $E_t$ extracted by $\phi$ are toroidal divisors of $(X_t,\Sigma_{X_t})$ and their centres on $X_t$ are strata of $\Sigma^h_{X_t}$. Since every stratum of $\Sigma^h_X$ has irreducible fibres over $T$ by \textbf{Step 2}, every stratum of $\Sigma^h_{X_t}$ arises as the intersection of $X_t$ with a unique stratum of $\Sigma^h_X$. Thus, we can find a sequence of toroidal extractions of strata of $\Sigma^h_X$, denoted by $p:X'\to X$, which induces $\phi$. In particular, we can identify $Y$ with  the fibre $X'_t$. Moreover, there is a unique divisor $\Omega$ on $X'$ supported on the toroidal boundary $\Sigma_{X'}$ of $X'$ such that $C=\Omega_t$. By Lemma \ref{extract hor wss}, $f':(X',\Sigma_{X'})\to (Z,\Sigma_Z)/T$ is still a family of weak semistable morphisms with $0\le\Omega\le \Sigma^h_{X'}$. Let $\CF'$ be the foliation induced by $f'$. Then Theorem \ref{inv vol} implies that $$\vol(Y,K_\CG+C)=\vol(X'_t,K_{\CF'_t}+\Omega_t)=
 \vol(X'_0,K_{\CF'_0}+\Omega_0).$$
It follows that $$\{\vol(Y,K_{\CG}+C)\,|\,(Y,\CG,C)\in \CS\}=\{\vol(Y,K_{\CG}+C)\,|\,(Y,\CG,C)\in \CS_0\}.$$
Now apply Proposition \ref{dcc of fixed model}.
\end{proof}

\vspace{0.6cm}
\section*{Acknowledgements}
The author expresses gratitude to his advisor Chen Jiang for great support and encouragement. He is grateful to Jingjun Han and Lingyao Xie for answering numerous questions and giving valuable suggestions. He would like to thank Paolo Cascini for his feedback. He would like to thank Dan Abramovich for answering questions on toroidal geometry. He would also like to thank Ziqi Liu, Pengjin Wang, and Qingyuan Xue for useful discussions.

This work was supported by National Key Research and Development Program
of China (No. 2023YFA1010600) and NSFC for Innovative
Research Groups (No. 12121001).
\vspace{0.6cm}
\printbibliography

\end{document}